\documentclass[12pt,reqno]{amsart}
\usepackage{latexsym,amsmath,amssymb, amsthm, mathscinet}
\usepackage{cases, verbatim}

\usepackage{amsfonts,amsmath,amsthm, amssymb}
\usepackage{cases}
\usepackage[usenames]{color}
\usepackage{enumerate}
\usepackage{graphicx}
\usepackage{bmpsize}
\newtheorem{thm}{Theorem}[section]
\newtheorem{prop}{Proposition}[section]
\newtheorem{lemma}{Lemma}[section]
\newtheorem{defn}{Definition}[section]

\newtheorem{remark}{Remark}[section]

\numberwithin{equation}{section}
\newcommand{\N}{\mathbb{N}}

\newcommand{\R}{\mathbb{R}}

\newcommand{\AC}{{\rm AC\,}}
\newcommand{\AL}{{\rm AL\,}}
\newcommand{\SP}{{\rm SP\,}}

\newcommand{\USC}{{\rm USC\,}}
\newcommand{\LSC}{{\rm LSC\,}}

\newcommand{\Lip}{{\rm Lip\,}}

\newcommand{\ep}{\varepsilon}

\newcommand{\ol}{\overline}

\newcommand{\pl}{\partial}

\newcommand{\inter}{{\rm int}\,}

\begin{document}
\title[Optimal semiconcavity with fractional modulus]{Optimal semiconcavity with fractional modulus \\ for Hamilton--Jacobi equations \\ with Neumann boundary conditions}

\author{Hiroyoshi Mitake}
\address[H. Mitake]{
	Department of Mathematics, 
	Faculty of Science and Engineering, 
	Waseda University, 
	3-4-1 Okubo, Shinjuku-ku, Tokyo, 169-8555, Japan}
\email{mitake@waseda.jp}
%\thanks{}

\author{Panrui Ni}
\address[P. Ni]{
Department 1: Department of Mathematics, Faculty of Science and Engineering, Waseda University, 3-4-1 Okubo, Shinjuku-ku, Tokyo, 169-8555, Japan; Department 2: Shanghai Center for Mathematical Sciences, Fudan University, Shanghai 200438, China}
\email{panruini@gmail.com}

%\thanks{}

\makeatletter
\@namedef{subjclassname@2020}{\textup{2020} Mathematics Subject Classification}
\makeatother

\date{\today}
\keywords{Hamilton--Jacobi equations, Viscosity solutions, Neumann boundary conditions, Semiconcavity}
\subjclass[2020]{35F21, 49L25, 35B65}

\begin{abstract}
We study the semiconcavity property of viscosity solutions to Hamilton--Jacobi equations with Neumann boundary conditions. Unlike the state-constraint case, minimizing trajectories associated with the Neumann problem may fail to be $C^1$, so the classical approach based on the regularity of minimizers is no longer available. To overcome this difficulty, we introduce a comparison argument between the constrained action associated with the Skorokhod problem and the unconstrained action, avoiding any use of higher regularity of reflected minimizing trajectories. Under a structural decomposition assumption on the Hamiltonian at the boundary, we establish the estimate
\[u(x+h,t+\sigma)+u(x-h,t-\sigma)-2u(x,t)\leq C(|h|+\sigma)^{\frac{3}{2}}.\]
An explicit example shows that the power $3/2$ in this estimate cannot be improved.
\end{abstract}

\date{\today}

\maketitle

\section{Introduction}

Going beyond the well-posedness theory for viscosity solutions is a fundamental issue in the study of Hamilton--Jacobi equations. For general coercive Hamiltonians, the classical viscosity solution theory ensures, under natural assumptions on the initial data, at most Lipschitz regularity of solutions, cf. \cite[Theorem 1.32]{T}. When the Hamiltonian is convex in the momentum variable, the situation becomes substantially richer. In this case, the associated optimal control representation provides additional structure, and yields local semiconcavity of solutions under standard regularity assumptions, cf. \cite[Theorem 5.3.8]{CS}. This regularity is essentially optimal in general, since shocks and singularities can develop even from smooth initial data (see \cite{CWC,CS,CY,Y} and the references therein).

The semiconcavity is a natural extension of smoothness of functions since semiconcave functions can be obtained as envelopes of smooth functions, and plays an essential role in establishing uniqueness theories of generalized solutions even before the theory of viscosity solutions has been introduced. Moreover, such notions naturally appear in the optimal control theory, and indeed in the context of nonsmooth analysis and optimization, semiconcave functions have received much attention. 
It is well-known that interest in semiconcavity of functions initially begins in the classical literature of the study of nonlinear partial differential equations such as \cite{D,K}, which introduce several notions of generalized solutions to nonlinear partial differential equations.  

In the presence of boundary conditions, however, extending local semiconcavity to a global property becomes a delicate issue. The main difficulty lies in the analysis of optimal trajectories near the boundary, including whether they hit the boundary, how they are reflected or constrained, and how such interactions affect second-order one-sided estimates. 
Global semiconcavity results have been obtained in the state-constraint setting, in particular in the context of mean field games, see \cite{CCC1, CCMW} and related works (e.g. \cite{Han}). However, much less is known in the case of Neumann boundary conditions, where the dynamics is governed by a reflected system and the interaction with the boundary is more involved. To the best of our knowledge, global semiconcavity with fractional modulus for Neumann boundary value problems has not been established.

\medskip

In this paper we study Hamilton--Jacobi equations with Neumann boundary conditions of the form
\begin{equation}\label{eq:N}
\begin{cases}
u_t(x,t)+H(x,Du(x,t))=0\quad &\text{in }\Omega\times(0,\infty),
\\ Du(x,t)\cdot \nu(x)=g(x)\quad &\text{on }\partial\Omega\times (0,\infty),
\\ u(x,0)=u_0(x) \quad &\text{on }\ol{\Omega}, 
\end{cases}
\end{equation}
where $\Omega\subset\R^n$ with $n\in\N$ is an open (possibly unbounded) connected set with a $C^2$-boundary, 
$H=H(x,p):\R^n\times\R^n\to\R$ is a given $C^2$-function, $g$, $u_0:\R^n\to\R$ are given Lipschitz continuous functions, 
and $\nu(x)$ denotes the outward unit normal vector at $x\in\partial\Omega$. 
The function $u:\overline\Omega\times[0,\infty)\to\R$ is unknown, and we respectively denote by $u_t:=\frac{\partial u}{\partial t}$ and $Du:=(\frac{\partial u}{\partial x_1}, \cdots, \frac{\partial u}{\partial x_n})$ its time derivative and gradient with respect to the spatial variable. The main result of this paper establishes a global semiconcavity with fractional modulus for viscosity solutions of \eqref{eq:N} with optimal power $\frac{3}{2}$ in the estimate.

\medskip

Throughout the paper, we \textit{always} assume 
\begin{itemize}
\item [(A1)] $H\in C^2(\R^n\times \R^n)$, 
\[\lim_{|p|\to\infty}\inf\left\{\frac{H(x,p)}{|p|}\mid x\in \R^n\right\}=\infty,\]
for each $R>0$ there exists $C_R>0$ such that 
\[\|H\|_{C^2(\R^n\times B(0,R))}\leq C_R,\]
and there exists $\alpha_0\geq 1$ such that
\[\frac{1}{\alpha_0}I_n\leq  D^2_{pp}H(x,p)\leq \alpha_0 I_n\quad \text{for all} \ (x,p)\in\R^n\times\R^n,\]
where $I_n$ is the identity matrix of size $n$;  
\item [(A2)]$g, u_0\in \Lip(\R^n)$, and \[\|g\|_{\Lip(\R^n)}+\|Du_0\|_{L^\infty(\R^n)}<\infty,\]
where
\[\|g\|_{\Lip(\R^n)}:=\|g\|_{L^\infty(\R^n)}+\|Dg\|_{L^\infty(\R^n)}.\]
\end{itemize}

\medskip

We denote by $L$ the Legendre transform of $H$, that is, 
\[
L(x,v):=\sup_{p\in\R^n}(v\cdot p-H(x,p)) \quad\text{for all} \ (x,v)\in\R^n\times\R^n.
\]
By standard properties of the Legendre transform under {\rm (A1)}, $L\in C^2(\R^n\times \R^n)$, and for each $R>0$, there is $C_R>0$ such that 
\[
\|L\|_{C^2(\R^n\times B(0,R))}\leq C_R, 
\] 
and there exists $\alpha_1>0$ such that
\[
\alpha_1I_n\le D^2_{vv}L(x,v) 
\ \text{for all} \ (x,v)\in\R^n\times\R^n.
\]
Moreover, $v\mapsto L(x,v)$ is superlinear uniformly for $x\in\R^n$ as in (A1).

\medskip

Let us first recall the definition of viscosity solutions to \eqref{eq:N}. 
\begin{defn} 
{\rm(i)} 
Let $u\in\USC(\ol{\Omega}\times [0,\infty))$. The function $u$ is said to be a viscosity subsolution of \eqref{eq:N}
if $u(\cdot,0)\leq u_0$ on $\overline{\Omega}$, and, 
for any $\varphi\in C^1(\overline{\Omega}\times[0,\infty))$, 
if $(\hat{x},\hat{t})\in \overline{\Omega}\times(0,\infty)$ is a maximizer of $u-\varphi$, and if $\hat{x}\in\Omega$, 
then 
$$
\varphi_t(\hat{x},\hat{t})+H(\hat{x},D\varphi(\hat{x},\hat{t}))\leq 0;
$$
if $\hat{x}\in\partial\Omega$, then 
$$
\min\left\{
\varphi_t(\hat{x},\hat{t})+H(\hat{x},D\varphi(\hat{x},\hat{t})), 
D\varphi(\hat{x},\hat{t})\cdot \nu(\hat x)-g(\hat x)
\right\}\leq 0.
$$
{\rm(ii)}
Let $u\in\LSC(\ol{\Omega}\times [0,\infty))$. The function $u$ is said to be a viscosity supersolution of \eqref{eq:N} if $u(\cdot,0)\geq u_0$ on $\overline{\Omega}$, and, 
for any $\varphi\in C^1(\overline{\Omega}\times[0,\infty))$, 
if $(\hat{x},\hat{t})\in \overline{\Omega}\times(0,\infty)$ is a minimizer of $u-\varphi$, and if $\hat{x}\in\Omega$, 
then 
$$
\varphi_t(\hat{x},\hat{t})+H(\hat{x},D\varphi(\hat{x},\hat{t}))\geq 0;
$$
if $\hat{x}\in\partial\Omega$, then 
$$
\max\left\{
\varphi_t(\hat{x},\hat{t})+H(\hat{x},D\varphi(\hat{x},\hat{t})), 
D\varphi(\hat{x},\hat{t})\cdot \nu(\hat x)-g(\hat x)
\right\}\geq 0.
$$
{\rm(iii)} 
A continuous function $u$ is said to be a viscosity solution of \eqref{eq:N} if $u$ is a viscosity subsolution and supersolution of \eqref{eq:N}. 
\end{defn}

We are \textit{always} concerned with viscosity solutions,
and the adjective ``viscosity" is often omitted in the paper. 
For each $T\ge 0$, denote by $\AL(\overline{\Omega}\times [0,T])$ the set of functions $w:\overline{\Omega}\times [0,T]\to\R$ that grow at most linearly in $x$, 
that is, $w\in \AL(\overline{\Omega}\times [0,T])$ if there exists a constant $C_T=C_T(w)>0$ such that 
\begin{equation}\label{cond:growth}
|w(x,t)|\leq C_T(1+|x|) \quad \text{ for $(x,t)\in \overline{\Omega}\times [0,T]$.}
\end{equation}
Set
\[
\AL(\overline{\Omega}\times [0,\infty)) = \left\{w:\overline{\Omega}\times [0,\infty)\to\R \mid w\in \AL(\overline{\Omega}\times [0,T]) \text{ for all } T>0 \right\}.
\]
In this paper, we assume that $\Omega$ is either a bounded domain, an exterior domain (i.e., $\Omega=\R^n\setminus K$ for some compact set $K\subset \R^n$), or a periodic half-space type domain. When $\Omega$ is a bounded, open connected subset of $\R^n$ with $C^1$ boundary, the uniqueness of the solution to \eqref{eq:N} was established in \cite[Theorem 3.4]{I11} and \cite[Theorem 3.1]{I13}. If $\Omega$ is an exterior domain, the comparison principle for \eqref{eq:N} can be proved as follows: for $x$ near $\partial\Omega$, one argues as in \cite[Theorem 3.4]{I11}, while for $x$ far from $\partial\Omega$, one follows the argument in \cite[Theorem 5.1]{guide} (see also \cite[Appendix A]{MNT}). This yields the uniqueness of the solution to \eqref{eq:N} in $\AL(\overline{\Omega}\times [0,\infty))$. The example in Section \ref{sec:ex1} concerns the exterior domain case. When $\Omega$ is a periodic half-space type domain, we refer the results in \cite{BDLS}. Both examples in Section \ref{sec:half} concern the half-space case. For possibly unbounded domains $\Omega$, we also refer the readers to \cite{Day,L}. We also refer to \cite{BL} for the comparison principle under general Neumann type boundary conditions. Since the uniqueness of the solution to \eqref{eq:N} is not the central focus of this paper and its proof is standard, we omit the details here.

\medskip
To state our main results precisely, we impose a structural assumption on the Hamiltonian at the boundary, which decomposes it into tangential and normal components. We assume that 
\begin{itemize}
\item [(A3)] 
there exist $H_1\in C^2(\R^{n}\times\R^n)$, $H_2\in C^2(\R^{n}\times\R^n)$ such that 
\begin{align*}
&H(x,p)=H_1(x,p_\tau)+H_2(x,p_\nu)
&&\text{for all} \ 
(x,p)\in \partial \Omega\times\R^n,\\ 
&H_1(x,p)=H_1(x,p_\tau),\quad H_2(x,p)=H_2(x,p_\nu)
&& \text{for all} \ 
(x,p)\in \partial \Omega\times\R^n,\\
&D_{p_\nu}H_2(x,0)=0 &&\text{for all} \ x\in\partial\Omega
\end{align*}
where $p_\nu$ and $p_\tau$ are, respectively, the components of $p\in\R^n$ parallel and perpendicular to $\nu(x)$, 
that is, 
\[
p_\nu:=\big(p\cdot \nu(x)\big)\nu(x) \quad\text{and}\quad 
p_\tau:=p-\big(p\cdot\nu(x)\big)\nu(x),  
\]
and we denote by $D_{p_\nu}H_2(x,p)$ the directional derivative of $H_2(x,p)$ with respect to $p$ in the normal direction, that is,
$D_{p_\nu}H_2(x,p):=D_pH_2(x,p)\cdot \nu(x)$.
\end{itemize}
A typical example satisfying (A3) is the quadratic Hamiltonian $H(x,p)=a(x)|p|^2+b(x)$, which naturally splits into tangential and normal parts. Here, $a(x)$ and $b(x)$ are two bounded $C^2$ functions and $\inf_{x\in\R^n}a(x)>0$. 

\medskip

We note that $g$ is defined on $\R^n$ in our setting. We set \[\Omega_+:=\{x\in\ol \Omega\mid g(x)\geq 0\}.\] Here is our main theorem. 
\begin{thm}\label{thm:main}
Assume {\rm (A1)--(A3)} hold. Let $u\in \mathrm{AL}(\overline{\Omega}\times [0,\infty))$ be the viscosity solution of \eqref{eq:N}, which is unique under the growth condition \eqref{cond:growth}. Then, for every $\varepsilon>0$, there exists a constant $C=C(\varepsilon,\|Du_0\|_{L^\infty(\R^n)},\|g\|_{\Lip(\R^n)},H,\Omega)>0$ such that
\[u(x+h,t+\sigma)+u(x-h,t-\sigma)-2u(x,t) \leq C (|h|+\sigma)^{\frac{3}{2}}\]
for all $(x,t)\in \Omega_+\times[\varepsilon,\infty)$, all $\sigma\geq 0$ satisfying $t-\sigma\geq 0$, and all $h\in\mathbb R^n$ such that $x\pm h\in \overline{\Omega}$. Moreover, the power $\frac{3}{2}$ in the estimate above is optimal.
\end{thm}

The restriction to $\Omega_+$ is essential in our argument. Indeed, in the estimates for the associated Skorokhod problem, one needs to control the additional reflection term $l(s)$ through a lower bound for $g(\eta(s))$ arising from the Neumann boundary condition. This requires the sign condition $g(x)\ge 0$, which is only available on $\Omega_+$; see Lemma \ref{SgeqA}. If only the values of $g$ on $\partial\Omega$ are prescribed, one may extend $g$ into $\R^n$ so that the set $\ol\Omega\setminus \Omega_+$ is localized near the set $\partial\Omega_-:=\{x\in\partial\Omega\mid g(x)<0\}$. However, achieving such a sharp transition generally requires a larger Lipschitz constant of the extension of $g$. Since the constant $C$ in Theorem \ref{thm:main} depends on $\|g\|_{\Lip(\R^n)}$, shrinking $\ol\Omega\setminus \Omega_+$ in this way comes at the expense of enlarging the constant in the estimate. When $\partial\Omega_-\neq \emptyset$, as illustrated by Section \ref{sec:4.1}, the negative boundary contribution may make it favorable for optimal trajectories to reach $\partial\Omega_-$ and remain there. This can reduce the value function near the boundary, and may even suggest the possibility of a downward cusp. At present, it is unclear whether semiconcavity can still be expected near $\partial\Omega_-$.

\medskip

Our proof of the main theorem, Theorem \ref{thm:main}, relies on the representation formula for the solution to \eqref{eq:N}, which is given by 
\begin{equation}\label{vf}
u(x,t)=\inf_{(\eta,v,l)\in \SP(x)}\left\{\int_0^tL(\eta(s),-v(s))+g(\eta(s))l(s)\, ds+u_0(\eta(t))\right\}.
\end{equation}
Here, we denote by $\SP(x)$ for $x\in\overline{\Omega}$ the family of solutions to 
the \textit{Skorokhod problem} (which is classically studied in \cite{LS}): for given $x\in\ol{\Omega}$ and $v\in L^1_{\rm{loc}}([0,\infty),\R^n)$, 
find a pair $(\eta,l)\in \AC_{\rm loc}([0,\infty),\R^n)\times L^1_{\rm loc}([0,\infty),[0,\infty))$ such that 
\begin{equation}\label{eq:Sk}
\begin{cases}
\eta(0)=x, \\
\eta(s)\in\ol{\Omega} &\text{for all} \ s\in [0,\infty),\\ 
\dot{\eta}(s)+l(s)\nu(\eta(s))=v(s),\quad & \text{for} \ a.e.\ s\in (0,\infty), \\ 
l(s)\ge 0 & \text{for} \ a.e.\ s\in (0,\infty), \\
l(s)=0\quad \textrm{if}\quad \eta(s)\in \Omega,\quad &\text{for} \ a.e.\ s\in (0,\infty). 
\end{cases}
\end{equation}
We refer to \cite[Theorem 4.1]{I11} for the existence result of solutions to \eqref{eq:Sk}. Since the proof of \cite[Theorem 5.1]{I11} is based on local arguments, \eqref{vf} gives a solution to \eqref{eq:N} even when $\Omega$ is unbounded. Moreover, although $\Omega$ may be unbounded, the representation formula \eqref{vf} still admits minimizers thanks to compactness properties of minimizing sequences established in \cite[Lemma 7.1 and Theorem 7.2]{I11}.

Since $u_0\in\Lip(\R^n)$, once the comparison principle has been established, we can prove the Lipschitz continuity of the value function $u(x,t)$ defined in \eqref{vf}, following \cite[Theorem 3.3]{I13}. Let $(\eta,v,l)\in\SP(x)$ be a minimizer of \eqref{vf}. Subsequently, by applying the arguments in \cite[Theorem 7.2]{I11} and \cite[Proposition 3.2]{MN1}, we deduce that $|\dot{\eta}|$, $l$, and $|v|$ are bounded on $[0,t]$ by a constant depending only on $H$, $\Omega$, $\|Du_0\|_{L^\infty(\R^n)}$, and $\|g\|_{L^\infty(\R^n)}$. Therefore, we only know that $\eta$ is Lipschitz continuous. For state-constraint problems, the semiconcavity estimates established in \cite{CCC1,CCMW} rely crucially on the $C^{1,1}$-regularity of minimizing trajectories. Such a regularity was obtained in \cite{CCC2} by means of a penalization argument and the Pontryagin maximum principle. For the Neumann problem, however, the situation is fundamentally different. As shown by the examples in Section \ref{sec:half}, minimizing trajectories associated with the Skorokhod problem may fail to be $C^1$, even in very simple situations. Consequently, the regularity theory developed for state-constraint problems cannot be directly extended to the Neumann setting. The main difficulty comes from the reflection mechanism encoded by the term $l(s)\nu(\eta(s))$, which may create discontinuities in the velocity of minimizing trajectories.

Assumption {\rm (A3)} allows us to separate the tangential and normal components of the dynamics and to control the contribution of the reflection term. Nevertheless, unlike the state-constraint case, our proof of Theorem \ref{thm:main} does not rely on any higher-order regularity of minimizing trajectories. Instead, we compare the constrained action associated with the Skorokhod problem with the corresponding unconstrained classical action. More precisely, we show that the difference between the constrained and unconstrained action functionals is of order $O(\tau^3)$ for short time $\tau$. This allows us to transfer the classical semiconcavity estimates for the unconstrained action directly to the constrained setting, without relying on any regularity of minimizing trajectories. This comparison yields the optimal $\frac{3}{2}$-estimate in Theorem \ref{thm:main} and provides a new approach to semiconcavity for Hamilton--Jacobi equations with Neumann boundary conditions. We also observe in Remark \ref{rmk1} that the same comparison argument between constrained and unconstrained actions may be useful in the study of state-constraint problems.

\bigskip
\noindent
\noindent \textbf{Notation.} For $X\subset\R^n$, $\Lip(X)$ (resp., $\USC(X)$, and $\LSC(X)$) denotes the space of Lipschitz continuous (resp., upper semicontinuous and lower semicontinuous) functions on $X$ with values in $\R$. For $A\subset \R^n$, $B\subset \R^m$ with $n,m\in\mathbb N$, we denote by $\mathrm{AC}(A,B)$ the family of absolutely continuous functions on $A$ with values in $B$. 
The set $B(x,r)\subset \R^n$ stands for the open ball centered at $x$ with the radius $r$. For any vector $a\in\R^n$ and $x\in\pl\Omega$, we define $a_\nu:=\big(a\cdot \nu(x)\big)\nu(x)$ and $a_\tau:=a-\big(a\cdot\nu(x)\big)\nu(x)$. 

\bigskip

\noindent
\textbf{Organization. } In Section \ref{sec2}, we establish an interior semiconcavity property of viscosity solutions to \eqref{eq:N}  in Proposition \ref{prop:inner-semi}. Section \ref{sec5} is devoted to the proof of Theorem \ref{thm:main}. In Section \ref{sec:half}, we will provide two examples showing that $C^1$-regularity of the minimizers of \eqref{vf} may fail, which is a main difference between the state constraint problem and the Neumann boundary problem. In Section \ref{sec:ex1}, we provide an example demonstrating the optimality of the fractional exponent in Theorem \ref{thm:main}.

\section{Interior Semiconcavity}\label{sec2}
In this section, we give a proof of an interior semiconcavity property of viscosity solutions to \eqref{eq:N}. Although such estimates are standard, the authors are not aware of a corresponding result for the Neumann boundary value problem. We therefore include the proof here for the reader's convenience, and also to emphasize a delicate difference and difficulty coming from the boundary. In the interior, minimizing trajectories satisfy the classical Euler--Lagrange equation, and the associated momentum defined in \eqref{func:p} is Lipschitz continuous. This property plays an essential role in the proof below. As shown in Section \ref{sec:half}, such regularity is no longer available in general when boundary reflections occur. To state a result of an interior semiconcavity property, we let $\rho>0$, and we set 
\[\Omega_\rho:=\{x\in\Omega\mid \ d_{\partial\Omega}(x)> \rho\},\]
where we set 
$d_K(x):=\inf_{y\in\ol K}|x-y|$ for a set $K\subset\R^n$. Let $(\eta,v,l)\in \SP(x)$ be a minimizer of \eqref{vf}. By \cite[Theorem 7.2]{I11}, there exists $C_0>0$ depending on $\|Du_0\|_{L^\infty(\R^n)}$, $\|g\|_{L^\infty(\R^n)}$, $H$ and $\Omega$ such that \[|\dot{\eta}(s)|\le C_0\quad \text{for a.e. }s\in(0,t).\]

\begin{prop}[Interior Semiconcavity] \label{prop:inner-semi}
Assume that {\rm (A1), (A2)} hold. For $\ep, \rho>0$, there exists a constant $c_{\ep,\rho}\ge 1$ depending on $\ep$ and $\rho$ such that 
\[
u(x+h,t+\sigma)+u(x-h,t-\sigma)-2u(x,t)\le c_{\ep,\rho}(|h|+\sigma)^2
\]
for all $(x,t)\in \Omega_\rho\times [\ep,\infty)$, $h\in\R^n$ with $|h|\leq\frac{\rho}{8}$ and $\sigma\in [0,\min(\frac{\ep}{2},\frac{\rho}{16C_0}))$. 
\end{prop}

\begin{lemma}\label{lem:ix}
For any $\ep, \rho>0$, there exists a constant $c_{\ep,\rho}\geq 1$ depending on $\ep$ and $\rho$ such that for any minimizer $(\eta,v,l)\in \SP(x)$ of \eqref{vf}, we have 
\[
u(x+h,t)-u(x,t)- p(0)\cdot h\leq c_{\ep,\rho}|h|^2,
\]
for all $(x,t)\in \Omega_\rho\times [\ep,\infty)$ and $h\in\R^n$ with $|h|\leq \frac{\rho}{4}$, 
where $p(s)$ is given by
\begin{equation}\label{func:p}
p(s):=D_vL(\eta(s),-v(s)). 
\end{equation}
\end{lemma}
\begin{proof}
Set 
\begin{align*}
&r:=\min\left(\frac{\ep}{2},\frac{\rho}{4C_0}\right), \nonumber \\
&\eta_h(s)=\eta(s)+\Big(1-\frac{s}{r}\Big)_+h \quad \text{for} \ s\in [0,\infty), 
\end{align*}
where $(s)_+:=\max(s,0)$ for $s\in\R$. 
We can easily see that $\eta_h(0)=x+h$ and $\eta_h(s)=\eta(s)$ for all $s\in [r,\infty)$. 
Then, for $s\in [0,r]$,
\begin{align*}
&|\eta(s)-\eta(0)|\leq C_0r\leq \frac{\rho}{4},\\ 
&|\eta_h(s)-\eta(0)|\leq |\eta_h(s)-\eta(s)|+|\eta(s)-\eta(0)|\leq |h|+C_0 r\leq \frac{\rho}{2},
\end{align*}
which implies
\begin{equation}\label{deta}
d_{\partial\Omega}(\eta(s))\geq \frac{3\rho}{4}\quad \text{and}\quad d_{\partial\Omega}(\eta_h(s))\geq \frac{\rho}{2}\quad \text{for }s\in [0,r].
\end{equation}
Thus, $\eta(s), \eta_h(s)\in\Omega$ for all $s\in[0,r]$. 
We take $l_h(s):=0$ for $s\in [0,r]$, and $l_h(s)=l(s)$ for all $s\geq r$. 
Set $v_h(s):=\dot \eta_h(s)+\nu(\eta_h(s))l_h(s)$. Then, $(\eta_h,v_h,l_h)\in\SP(x+h)$. 
Therefore, by the dynamic programming principle, 
\begin{align*}
&u(x+h,t)-u(x,t)-p(0)\cdot h
\\ &\leq u(\eta_h(r),t-r)+\int_0^r L(\eta_h(s),-\dot{\eta}_h(s))\, ds
\\ &\qquad -u(\eta(r),t-r)-\int_0^rL(\eta(s),-\dot\eta(s))\, ds- p(0)\cdot h
\\ &=\int_0^r\Big(L(\eta_h(s),-\dot{\eta}_h(s))-L(\eta(s),-\dot\eta(s))\Big)\, ds- p(0)\cdot h,
\end{align*}
where we used the fact that $\eta_h(r)=\eta(r)$.

For $s\in[0,r]$, since $\eta(s)\in\Omega$, we have 
$l(s)=0$. 
Therefore, the classical derivation of the Hamiltonian system implies that $p$ is Lipschitz continuous in $(0,r)$ and satisfies  
\begin{equation}\label{eq:peta}
\dot{p}(s)=D_xH(\eta(s), p(s)), \quad p(s)=D_vL(\eta(s),-\dot{\eta}(s)) \quad\text{for} \ a.e. \ s\in(0,r). 
\end{equation}
Then, we have
\begin{align*}
-p(0)\cdot h&=-p(r)\cdot (\eta_h(r)-\eta(r))+\int_0^r\frac{d}{ds}\Big(p(s)\cdot (\eta_h(s)-\eta(s))\Big)\, ds\\ &=\int_0^r\Big(\dot{p}(s)\cdot (\eta_h(s)-\eta(s))+p(s)\cdot (\dot {\eta}_h(s)-\dot \eta(s))\Big)\, ds\\
&=\int_0^r\Big(-D_xL(\eta(s),-\dot\eta(s))\cdot (\eta_h(s)-\eta(s))+D_vL(\eta(s),-\dot\eta(s))\cdot (\dot {\eta}_h(s)-\dot \eta(s))\Big)\, ds. 
\end{align*}
Thus, we have 
\begin{align*}
&u(x+h,t)-u(x,t)-p(0)\cdot h
\\ &\le \int_0^r\Big(L(\eta_h(s),-\dot{\eta}_h(s))-L(\eta(s),-\dot\eta(s))\, ds
\\ &\quad +\int_0^r\Big(-D_xL(\eta(s),-\dot\eta(s))\cdot (\eta_h(s)-\eta(s))+D_vL(\eta(s),-\dot\eta(s))\cdot (\dot {\eta}_h(s)-\dot \eta(s))\Big)\, ds
\\ &=\int_0^r\Big(L(\eta_h(s),-\dot{\eta}_h(s))-L(\eta(s),-\dot{\eta}_h(s))-D_xL(\eta(s),-\dot{\eta}_h(s))\cdot (\eta_h(s)-\eta(s))\Big)\, ds
\\ &\quad +\int_0^r \Big(L(\eta(s),-\dot{\eta}_h(s))-L(\eta(s),-\dot\eta(s))+D_vL(\eta(s),-\dot\eta(s))\cdot (\dot{\eta}_h(s)-\dot \eta(s))\Big)\, ds
\\ &\quad +\int_0^r \Big(D_xL(\eta(s),-\dot{\eta}_h(s))\cdot (\eta_h(s)-\eta(s))-D_xL(\eta(s),-\dot\eta(s))\cdot (\eta_h(s)-\eta(s))\Big)\, ds
\\ &\leq C\int_0^r\Big(|\eta_h(s)-\eta(s)|^2+|\dot{\eta}_h(s)-\dot \eta(s)|^2+|\dot{\eta}_h(s)-\dot \eta(s)||\eta_h(s)-\eta(s)|\Big)\, ds.
\end{align*}
Noting that 
\[
|\eta_h(s)-\eta(s)|\leq |h|\quad \text{and}\quad |\dot{\eta}_h(s)-\dot \eta(s)|=\frac{|h|}{r} 
\quad\text{for all} \ s\in[0,r], 
\]
we conclude
\[
u(x+h,t)-u(x,t)-p(0)\cdot h\leq C\Big(|h|^2r+\frac{|h|^2}{r}+|h|^2\Big)\leq c_{\ep,\rho} |h|^2,
\]
since $r$ is fixed here.
\end{proof}

\begin{lemma}\label{lem:inner2}
For any $\ep, \rho>0$, there exists a constant $c_{\ep,\rho}\geq 1$ depending on $\ep$ and $\rho$ such that for any minimizer $(\eta,v,l)\in \SP(x)$ of \eqref{vf}, we have 
\[
u(x+h,t-\sigma)-u(x,t)\leq p(0)\cdot h+\sigma H(x,p(0)) +c_{\ep,\rho}(|h|+\sigma)^2,
\]
for all $(x,t)\in \Omega_\rho\times [\ep,\infty)$, all $h\in\R^n$ with $|h|\leq \frac{\rho}{8}$, 
and all $\sigma\in(0,min(\frac{\ep}{2},\frac{\rho}{16C_0})]$, where $p(s)$ is given by \eqref{func:p}.  
\end{lemma}

\begin{proof}
Note first that $|\eta(s)-x|\leq C_0\sigma\leq \frac{\rho}{16}$, which implies 
\[d_{\partial\Omega}(\eta(s))> \frac{15\rho}{16}, \quad \text{and} \quad
\eta(s)\in\Omega \quad\text{for all} \ s\in [0,\sigma]. 
\]
By the dynamic programming principle, we have 
\begin{align*}
&u(x+h,t-\sigma)-u(x,t)
\\ &= u(x+h,t-\sigma)-u(\eta(\sigma),t-\sigma)-\int_0^\sigma L(\eta(s),-\dot\eta(s))\, ds.
\end{align*}
Since $\eta(\sigma)\in \Omega_{\frac{15\rho}{16}}$, by Lemma \ref{lem:ix}, there is $c_{\ep,\rho}\geq 1$ such that
\[u(x+h,t-\sigma)-u(\eta(\sigma),t-\sigma)\leq p(\sigma)\cdot (x+h-\eta(\sigma))+c_{\ep,\rho}|x+h-\eta(\sigma)|^2.\]
Note that 
\[
|x+h-\eta(\sigma)|\leq |h|+|\eta(0)-\eta(\sigma)|\leq |h|+C_0\sigma\leq \frac{3\rho}{16}<\frac{15\rho}{64}.
\]
As seen in the proof of Lemma \ref{lem:ix}, $p$ is Lipschitz continuous in $(0,\sigma)$ since 
$\eta(s)\in\Omega_{\frac{\rho}{2}}$ for all $s\in(0,\sigma)$. 
Therefore, 
\begin{align*}
&p(\sigma)\cdot (x+h-\eta(\sigma))
=(p(\sigma)-p(0))\cdot h+p(0)\cdot h-\int_0^\sigma p(\sigma)\cdot\dot\eta(s)\, ds
\\ \leq&\, 
C\sigma|h|+p(0)\cdot h-\int_0^\sigma p(s)\cdot\dot\eta(s)\, ds+C\sigma^2
\end{align*}
for some $C\ge0$. 
Thus, 
\begin{align*}
&u(x+h,t-\sigma)-u(x,t)
\\ &\leq\, p(0)\cdot h-\int_0^\sigma \Big(p(s)\cdot \dot\eta(s)+L(\eta(s),-\dot\eta(s))\Big)\, ds
+C|h|^2+C\sigma^2+c_{\ep,\rho} (|h|+\sigma)^2.
\end{align*}
Moreover, noting that 
\[
-p(s)\cdot\dot\eta(s)-L(\eta(s),-\dot\eta(s))=H(\eta(s),p(s))\quad \text{for }s\in [0,\sigma], 
\]
and 
\[
H(\eta(s),p(s))\leq H(\eta(0),p(0))+C\sigma
\]
we obtain the conclusion.
\end{proof}

\begin{lemma}\label{lem:inner3}
For any $\ep, \rho>0$, there exists a constant $c_{\ep,\rho}\geq 1$ depending on $\ep$ and $\rho$ such that for any minimizer $(\eta,v,l)\in \SP(x)$ of \eqref{vf}, we have 
\[
u(x+h,t+\sigma)-u(x,t)\leq p(0)\cdot h-\sigma H(x,p(0)) +c_{\ep,\rho}(|h|+\sigma)^2,
\]
for all $(x,t)\in \Omega_\rho\times [\ep,\infty)$, all $h\in\R^n$ with
$|h|\leq \frac{\rho}{8}$ and all $\sigma\in (0,\min(\frac{\ep}{2},\frac{\rho}{8C_0}))$, 
where $p(s)$ is given by \eqref{func:p}. 
\end{lemma}

\begin{proof}
Let $r$ and $\eta_h$ be defined as in the proof of Lemma \ref{lem:ix}. 
Then, \eqref{deta} holds.
Set 
\begin{equation*}
\eta_{h,\sigma}(s):=
\begin{cases}
\eta_h(s)&\text{for} \quad s\in [0,\sigma],\\ 
\eta_h(\sigma)&\text{for} \quad s\in [\sigma,\infty).
\end{cases}
\end{equation*}
Then $(\eta_{h,\sigma},\dot\eta_{h,\sigma},0)\in\SP(x+h)$. Since $\sigma<r$, we have $\eta_{h,\sigma}(s)=\eta(s)+(1-\frac{s}{r})h$ for $s\in [0,\sigma]$. By the dynamic programming principle we obtain
\begin{align*}
&u(x+h,t+\sigma)-u(\eta_{h,\sigma}(\sigma),t)
\\ &\leq \int_0^{\sigma}L(\eta_{h,\sigma}(s),-\dot\eta_{h,\sigma}(s))\, ds=\int_0^{\sigma}L(\eta_{h}(s),-\dot{\eta}_{h}(s))\, ds
\\ &\leq \int_0^{\sigma}\Big(L(\eta(s),-\dot\eta(s))+D_vL(\eta(s),-\dot\eta(s))\cdot (\dot\eta(s)-\dot{\eta}_h(s))
\\ &\qquad \qquad +C|\dot{\eta}_h(s)-\dot\eta(s)|^2\Big)\, ds+C\sigma |h|
\\ &\leq \int_0^{\sigma}\Big(L(\eta(s),-\dot\eta(s))+D_vL(\eta(s),-\dot\eta(s))\cdot (\dot\eta(s)-\dot{\eta}_h(s))\Big)\, ds+C\sigma \frac{|h|^2}{r^2}+C\sigma |h|,
\end{align*}
where $\sigma<1$ and $r$ is fixed and depends on $\ep$ and $\rho$. Recalling \eqref{eq:peta} and the fact that $p$ is Lipschitz continuous in $(0,\sigma)$, we obtain
\begin{align*}
&-\int_0^\sigma D_vL(\eta(s),-\dot\eta(s))\cdot \dot{\eta}_h(s)\, ds
=-\int_0^\sigma p(s)\cdot \dot{\eta}_h(s)\, ds\\
\leq&\, 
-p(0)\cdot \int_0^\sigma \dot{\eta}_h(s)\, ds+C\sigma^2
=-p(0)\cdot\big(\eta_h(\sigma)-x-h\big)+C\sigma^2. 
\end{align*}
On the other hand, by Lemma \ref{lem:ix}, we have
\[
u(\eta_{h,\sigma}(\sigma),t)-u(x,t)\leq p(0)\cdot (\eta_{h}(\sigma)-x)+C|\eta_{h}(\sigma)-x|^2,
\]
where 
\[
|\eta_{h}(\sigma)-x|\leq |\eta_h(\sigma)-\eta(\sigma)|+|\eta(\sigma)-x|\leq |h|+C_0\sigma\leq \frac{\rho}{4}.
\]
Thus, 
\begin{align*}
&u(x+h,t+\sigma)-u(x,t)
\\ &\leq p(0)\cdot h+\int_0^{\sigma}\Big(L(\eta(s),-\dot\eta(s))+D_vL(\eta(s),-\dot\eta(s))\cdot \dot\eta(s)\Big)\, ds
+C(\sigma+|h|)^2.
\end{align*}
Noting that 
\[
D_vL(\eta(s),-\dot\eta(s))\cdot \dot\eta(s)+L(\eta(s),-\dot\eta(s))=-H(\eta(s),p(s)), 
\]
and
\[
-\int_0^{\sigma}H(\eta(s),p(s))\, ds\leq -\sigma H(x,p(0))+C\sigma^2,
\]
we obtain the conclusion. 
\end{proof}

\medskip

\begin{proof}[Proof of Proposition {\rm\ref{prop:inner-semi}}]
By Lemmas \ref{lem:inner2}, \ref{lem:inner3}, 
for any $\ep, \rho>0$, there exists a constant $c_{\ep,\rho}\geq 1$ depending on $\ep$ and $\rho$ such that 
for all $(x,t)\in \Omega_\rho\times [\ep,\infty)$, all $h\in\R^n$ with 
$|h|\leq \frac{\rho}{8}$ and all $\sigma\in [0,\min(\frac{\ep}{2},\frac{\rho}{16C_0}))$,  
\begin{align*}
&u(x+h,t+\sigma)+u(x-h,t-\sigma)-2u(x,t)\\
=&\, 
u(x+h,t+\sigma)-u(x,t)-p(0)\cdot h-\sigma H(x,p(0))\\
&\, +u(x-h,t-\sigma)-u(x,t)-p(0)\cdot(-h)+\sigma H(x,p(0))
\le c_{\ep,\rho}(|h|+\sigma)^2, 
\end{align*}
where $p(s)$ is given by \eqref{func:p} for any minimizer $(\eta,v,l)\in \SP(x)$ of \eqref{vf}. 
\end{proof}

\section{Fractional semiconcavity}\label{sec5}

In this section, we give a proof of Theorem \ref{thm:main}. 
We first give an elementary result based on a simple observation. Let $(\eta,v,l)\in \SP(x)$ be a minimizer of \eqref{vf} and let $C_0$ be a positive constant satisfying $|\dot{\eta}(s)|\le C_0$ for a.e. $s\in(0,t)$. Define
\[I_0:=\{s\in[0,t]\mid \eta(s)\in \Omega\},\quad I_0^c:=[0,t]\setminus I_0.\]

\begin{lemma}\label{lem:dot}
For a.e. $s\in I^c_0$, we have
\[
\nu(\eta(s))\cdot\dot\eta(s)=0.
\]
\end{lemma}

\begin{proof}
Let us take a function $\psi\in C^1(\mathbb R^n)$ satisfying  
\begin{align*}
&\Omega=\{x\in\mathbb R^n\mid \psi(x)<0\}, 
\quad 
\partial\Omega=\{x\in\mathbb R^n\mid \psi(x)=0\}, \nonumber\\
&D\psi(x)=|D\psi(x)|\nu(x) \ \textrm{for all} \ x\in\partial\Omega. 
\end{align*}
For $s_0\in I_0^c$, we have
\[\eta(s_0)\in \partial\Omega,\quad \psi(\eta(s_0))=\max_{s\in[0,t]}\psi(\eta(s))=0.\]
Since $\eta\in \AC([0,t],\ol\Omega)$, if $\eta$ is differentiable at $s_0$, then we have
\[
0=\frac{d}{ds}\bigg|_{s=s_0}\psi(\eta(s))=D\psi(\eta(s_0))\cdot \dot\eta(s_0)=|D\psi(\eta(s_0))|\nu(\eta(s_0))\cdot \dot\eta(s_0),
\]
which implies 
\begin{equation*}
\nu(\eta(s_0))\cdot \dot\eta(s_0)=0.
\end{equation*} 
\end{proof}

For $\tau>0$, $x,y\in \ol\Omega$, consider the constrained action associated with the Skorokhod problem
\[S_\tau(x,y)=\inf \Big\{\int_0^\tau \Big(L(\eta(s),-v(s))+g(\eta(s))l(s)\Big)\, ds\mid (\eta,v,l)\in \SP(x),\ \eta(\tau)=y\Big\},\]
and the corresponding unconstrained classical action
\[A_\tau(x,y)=\inf \Big\{\int_0^\tau L(\eta(s),-\dot\eta(s))\, ds\mid \eta\in \AC([0,\tau],\R^n),\ \eta(0)=x,\ \eta(\tau)=y\Big\}.\]
For the existence of minimizers of $S_\tau(x,y)$, we refer to \cite[Proposition 4.3]{MN1}. The existence of minimizers of $A_\tau(x,y)$ is ensured by the Tonelli theorem; see \cite{F}.
\begin{lemma}\label{SleqA}
Assume {\rm (A1)}. For every $C_0>0$, there exist constants $\tau_0>0$ and $C>0$, depending only on $C_0$, $H$, and $\Omega$, such that for every $\tau\in(0,\tau_0]$ and every $x,y\in\ol \Omega$ with $|x-y|\leq C_0\tau$,
\[S_\tau(x,y)\leq A_\tau(x,y)+C\tau^3.\]
\end{lemma}
\begin{proof}
Let $\eta:[0,\tau]\to\R^n$ be a minimizer of $A_\tau(x,y)$. We take a straight line
\[\alpha(s)=x+\frac{y-x}{\tau}s,\quad s\in [0,\tau].\]
Since $|x-y|\leq C_0\tau$, $|\dot\alpha(s)|\leq C_0$. By $D^2_{vv}L\geq \alpha_1I_n$ and boundedness of $\|L\|_{C^2(\R^n\times B(0,R))}$, 
\begin{align*}
L(\eta(s),-\dot\eta(s))&\geq L(\eta(s),0)-D_v L(\eta(s),0)\cdot \dot\eta(s)+\frac{\alpha_1}{2}|\dot\eta(s)|^2
\\ &\geq |\dot\eta(s)|-C_1\quad \text{for all }s\in [0,\tau],
\end{align*}
for some $C_1>0$. Hence,
\begin{align*}
\int_0^\tau\Big(|\dot\eta(s)|-C_1\Big)\, ds&\leq \int_0^\tau L(\eta(s),-\dot\eta(s))\, ds
\\ &\leq \int_0^\tau L(\alpha(s),-\dot\alpha(s))\, ds\leq \sup_{(z,v)\in \R^n\times B(0,C_0)}L(z,v)\tau.
\end{align*}
Thus, there exists $s_0\in [0,\tau]$ such that $|\dot\eta(s_0)|$ is bounded by a constant depending only on $C_0$. Since $L$ is a Tonelli Lagrangian in our setting, $(\eta(s),\dot\eta(s))$ satisfies the Euler-Lagrange flow (cf. \cite[Theorems 2.6.4 and 3.4.2]{F}). Therefore, $\eta\in C^2([0,\tau],\R^n)$ and there exists $C>0$ depending on $C_0$ such that
\[|\dot\eta(s)|\leq C,\quad |\ddot\eta(s)|\leq C\quad \text{for all }s\in[0,\tau].\]
Let $\tau$ be small. If $d_{\partial\Omega}(x)\geq \sqrt{\tau}$, then $\eta(s)\in \ol\Omega$ for all $s\in[0,\tau]$, that is, $\eta$ is an admissible curve for $S_\tau$. We have
\[S_\tau(x,y)\leq \int_0^\tau L(\eta(s),-\dot\eta(s))\, ds \leq A_\tau(x,y).\]
We only need to take care of the case $d_{\partial\Omega}(x)<\sqrt{\tau}$. Since $\partial\Omega$ is $C^2$, and we assume that $\Omega$ is either a bounded domain, an exterior domain, or periodic half-space type domain, through a transformation, near $x$ we can write
\[\Omega\cap B(x,r)=\{(z',z_n)\in\R^{n-1}\times\R:\ z_n\geq f(z')\},\]
for some $r>0$ and $f\in C^2(\R^{n-1},\R)$ satisfying $\|f\|_{C^2(\R^{n-1})}\leq C$. We take $\tau>0$ sufficiently small so that $\eta(s)\in B(x,r)$ for all $s\in [0,\tau]$. Now we write $\eta(s)=(\eta'(s),\eta_n(s))$. Define
\[\rho(s):=\eta_n(s)-f(\eta'(s))\quad \text{for }s\in[0,\tau].\]
Since $x,y\in\ol \Omega$,
\[\rho(0)\geq 0,\quad \rho(\tau)\geq 0.\]
Calculating directly, we have
\begin{align*}
&\dot\rho(s)=\dot\eta_n(s)-D_{x'}f(\eta'(s))\cdot \dot{\eta}'(s),
\\ &\ddot{\rho}(s)=\ddot{\eta}_n(s)-D^2_{x'x'}f(\eta'(s))\dot\eta'(s)\cdot \dot{\eta}'(s)-D_{x'}f(\eta'(s))\cdot \ddot{\eta}'(s).
\end{align*}
Since $\dot\eta$, $\ddot\eta$ and $\|f\|_{C^2(\R^{n-1})}$ are bounded, we have
\[|\ddot\rho(s)|\leq C\quad \text{for }s\in[0,\tau].\]
Hence, by the comparison principle,
\begin{equation}\label{rhoge}
\rho(s)\geq -Cs(\tau-s),
\end{equation}
that is, $\eta$ can only leave $\ol\Omega$ by a distance of order $O(\tau^2)$.

\medskip

Now we construct an admissible curve for $S_\tau$. Define
\[\chi(s):=Cs(\tau-s)\quad \text{for }s\in[0,\tau].\]
By \eqref{rhoge},
\[\rho(s)+\chi(s)\geq 0.\]
Let $e_n$ denote the $n$-th standard basis vector of $\R^n$. Define
\[\tilde\eta(s):=\eta(s)+\chi(s)e_n\quad \text{for }s\in[0,\tau].\]
Then $\tilde\eta(s)\in B(x,r)$ for all $s\in[0,\tau]$. We write $\tilde\eta(s)=(\tilde\eta'(s),\tilde\eta_n(s))$. Then
\[\tilde\eta_n(s)=\eta_n(s)+\chi(s)=f(\eta'(s))+\rho(s)+\chi(s)\geq f(\eta'(s)),\]
that is, $\tilde\eta(s)\in\ol\Omega$. By definition of $\tilde\eta$, we have $\tilde\eta(0)=\eta(0)=x$, $\tilde\eta(\tau)=\eta(\tau)=y$, and 
\[|\eta(s)-\tilde\eta(s)|\leq C\tau^2,\quad |\dot{\tilde\eta}(s)-\dot\eta(s)|\leq C\tau\quad \text{for all }s\in [0,\tau].\]
We take $l(s)=0$ for all $s\geq 0$ and let $\tilde\eta(s)\equiv y$ for all $s\geq \tau$. Then $(\tilde\eta,\dot{\tilde\eta},0)\in \SP(x)$ and $\tilde\eta(\tau)=y$. We obtain
\begin{align*}
&\int_0^\tau \Big(L(\tilde\eta(s),-\dot{\tilde\eta}(s))-L(\eta(s),-\dot\eta(s))\Big)\, ds
\\ &\leq \int_0^\tau\Big(D_xL(\eta(s),-\dot\eta(s))\cdot(\tilde\eta(s)-\eta(s))+D_vL(\eta(s),-\dot\eta(s))\cdot(\dot\eta(s)-\dot{\tilde\eta}(s))
\\ &\qquad \qquad +C|\tilde\eta(s)-\eta(s)|^2+C|\dot{\tilde\eta}(s)-\dot\eta(s)|^2)\Big)\, ds.
\end{align*}
Integrating by parts, and noting that $\tilde\eta(0)=\eta(0)=x$ and $\tilde\eta(\tau)=\eta(\tau)=y$, so that the boundary terms vanish, we obtain from the Euler--Lagrange equation that
\begin{align*}
&\int_0^\tau\Big(D_xL(\eta(s),-\dot\eta(s))\cdot(\tilde\eta(s)-\eta(s))+D_vL(\eta(s),-\dot\eta(s))\cdot(\dot\eta(s)-\dot{\tilde\eta}(s))\Big)\, ds
\\ &=\int_0^\tau\Big(D_xL(\eta(s),-\dot\eta(s))+\frac{d}{ds}\big(D_vL(\eta(s),-\dot\eta(s))\big)\Big)\cdot(\tilde\eta(s)-\eta(s))\, ds=0,
\end{align*}
which implies
\[S_\tau(x,y)-A_\tau(x,y)\leq \int_0^\tau \Big(L(\tilde\eta(s),-\dot{\tilde\eta}(s))-L(\eta(s),-\dot\eta(s))\Big)\, ds\leq C\tau^3.\]
\end{proof} 

\begin{lemma}\label{SgeqA}
Assume {\rm (A1)--(A3)}. Let $\tau>0$. For $x\in\Omega_+$ and $y\in\ol\Omega$ with $|x-y|\leq C_0\tau$ for some $C_0>0$, if there exists a minimizer $(\eta,v,l)\in \SP(x)$ with $\eta(\tau)=y$ of $S_\tau(x,y)$ with $\eta$ being Lipschitz continuous with a Lipschitz constant $C_0$, then there exists $C>0$ depending on $C_0$, $\|Dg\|_{L^\infty(\R^n)}$ and $H$ such that
\[S_\tau(x,y)\geq A_\tau(x,y)-C\tau^3.\]
\end{lemma}
\begin{proof}
Under {\rm (A3)}, we define
\[v_\nu:=(v\cdot\nu(z))\nu(z),\quad v_\tau:=v-(v\cdot\nu(z))\nu(z)\quad\text{for all }z\in\partial\Omega.\]
Then $v_\nu\in\text{span}\{\nu(z)\}\simeq\R$ and $v_\tau\in T_z(\partial\Omega)\simeq \R^{n-1}$, where $T_z(\partial\Omega)$ is the tangent space of $\partial\Omega$ at $z$. We obtain 
\begin{align*}
L(z,v)
&=\sup_{p\in\R^n}\big\{(p_\tau+p_\nu)\cdot(v_\tau+v_\nu)-H_1(z,p_\tau)-H_2(z,p_\nu)\big\}\nonumber\\
&=L_1(z,v_\tau)+L_2(z,v_\nu) \quad\text{for all }z\in\partial\Omega,
\end{align*}
where we set for all $v\in\R^n$ and $z\in\partial\Omega$,
\begin{align*}
L_1(z,v)&:=\sup_{p \in\R^{n}}\big\{v\cdot p-H_1(z,p)\big\}=\sup_{p_\tau\in\R^{n-1}}(p_\tau\cdot v_\tau-H_1(z,p_\tau))+\sup_{p_\nu\in\R}p_\nu v_\nu
\\ &=\begin{cases}
\sup_{p_\tau\in\R^{n-1}}(p_\tau\cdot v_\tau-H_1(z,p_\tau))\quad & \text{for} \ v_\nu=0,\\
\infty\quad & \text{for} \ v_\nu\neq 0, 
\end{cases}
\\ L_2(z,v)&:=\sup_{p\in\R^n}\big\{v\cdot p-H_2(z,p)\big\}=\sup_{p_\nu\in\R}(p_\nu\cdot v_\nu-H_2(z,p_\nu))+\sup_{p_\tau\in\R^{n-1}}p_\tau v_\tau
\\ &=\begin{cases}
\sup_{p_\nu\in\R}(p_\nu\cdot v_\nu-H_2(z,p_\nu))\quad & \text{for} \ v_\tau=0,\\
\infty\quad & \text{for} \ v_\tau\neq 0.
\end{cases}
\end{align*}
By the duality of the Legendre transform and the assumption $D_{p_\nu}H_2(z,0)=0$, the directional derivative of $L_2(z,v)$ with respect
to $v$ in the normal direction satisfies
\begin{equation*}
D_{v_\nu}L_2(z,0)=0 \quad\text{for all} \ z\in\partial\Omega. 
\end{equation*} 
Here, for all $v\in\R^n$ and $z\in\partial\Omega$, 
\begin{align*}
D_{v_\nu}L_2(z,v_\nu)&:=\lim_{\epsilon\to 0}\frac{1}{\epsilon}\Big(L_2(z,v_\nu+\epsilon \nu(z))-L_2(z,v_\nu)\Big)
\\ &=\lim_{\ep\to 0}\frac{1}{\epsilon}\Big(L_1(z,v_\tau)+L_2(z,v_\nu+\epsilon \nu(z))-L_1(z,v_\tau)-L_2(z,v_\nu)\Big)
\\ &=\lim_{\ep\to 0}\frac{1}{\epsilon}\Big(L(z,v+\epsilon\nu(z))-L(z,v)\Big)=D_vL(z,v)\cdot\nu(z),
\end{align*}
and 
\begin{align*}
D^2_{v_\nu v_\nu}L_2(z,v_\nu)&:=\lim_{\epsilon\to 0}\frac{1}{\epsilon}\Big(D_{v_\nu}L_2(z,v_\nu+\epsilon \nu(z))-D_{v_\nu}L_2(z,v_\nu)\Big)
\\&= \lim_{\epsilon \to 0}\frac{1}{\epsilon}\Big(D_vL(z,v+\epsilon\nu(z))\cdot \nu(z)-D_vL(z,v)\cdot \nu(z)\Big)
\\ &=D^2_{vv}L(z,v)\nu(z)\cdot\nu(z)\geq \alpha_1.
\end{align*}
Let $(\eta,v,l)\in \SP(x)$ be a minimizer of $S_\tau(x,y)$ with $\eta$ being Lipschitz continuous with a Lipschitz constant $C_0$. For $s\in I^c_0$, we have
\[L_2(\eta(s),-l(s)\nu(\eta(s)))\geq L_2(\eta(s),0)+\frac{\alpha_1}{2}l(s)^2.\]
By Lemma \ref{lem:dot}, for a.e. $s\in I^c_0$, we have
\begin{align*}
&L(\eta(s),-\dot\eta(s)-l(s)\nu(\eta(s)))+g(\eta(s))l(s)
\\ &=L_1(\eta(s),-\dot\eta(s))+L_2(\eta(s),-l(s)\nu(\eta(s)))+g(\eta(s))l(s)
\\ &\geq L_1(\eta(s),-\dot\eta(s))+L_2(\eta(s),0)+\frac{\alpha_1}{2}l(s)^2+g(\eta(s))l(s)
\\ &=L(\eta(s),-\dot\eta(s))+\frac{\alpha_1}{2}l(s)^2+g(\eta(s))l(s).
\end{align*}
Here, 
\[\frac{\alpha_1}{2}l(s)^2+g(\eta(s))l(s)\geq -\frac{\max(-g(\eta(s)),0)^2}{2\alpha_1}.\]
Since $g(x)\geq 0$ and both $\eta$ and $g$ are Lipschitz continuous,
\[-g(\eta(s))\leq -g(x)+C\tau\leq C\tau,\]
for some constant $C>0$ depending on $\|Dg\|_{L^\infty(\R^n)}$ and $C_0$. Thus,
\begin{align*}
S_\tau(x,y) &=\int_0^\tau\Big(L(\eta(s),-v(s))+g(\eta(s))l(s)\Big)\, ds
\\ &\geq \int_0^\tau L(\eta(s),-\dot\eta(s))\, ds-C\tau^3\geq A_\tau(x,y)-C\tau^3,
\end{align*}
for some constant $C>0$ depending on $H$, $\|Dg\|_{L^\infty(\R^n)}$ and $C_0$. 
\end{proof}

\medskip

\noindent {\it Proof of Theorem \ref{thm:main}.} Assume $x\in \Omega_+$. Let $x\pm h\in \ol \Omega$. Let $(\eta,v,l)\in\SP(x)$ be a minimizer of \eqref{vf}. Define $y=\eta(r)$ for some small $r\in(0,\frac{\ep}{2}]$ to be determined. Since $\eta$ is Lipschitz continuous  with a Lipschitz constant $C_0$, we have $|x-y|\leq C_0r$. By the dynamic programming principle, for $t\geq \ep> r>\frac{r}{2}\geq \sigma\geq 0$, we have
\begin{align*}
&u(x,t)=u(y,t-r)+S_r(x,y),
\\ &u(x+h,t+\sigma)\leq u(y,t-r)+S_{r+\sigma}(x+h,y),
\\ &u(x-h,t-\sigma)\leq u(y,t-r)+S_{r-\sigma}(x-h,y).
\end{align*}
We get
\begin{align*}
&u(x+h,t+\sigma)+u(x-h,t-\sigma)-2u(x,t)
\\ &\qquad \leq S_{r+\sigma}(x+h,y)+S_{r-\sigma}(x-h,y)-2S_r(x,y).
\end{align*}
We take $|h|\leq r$, then
\begin{align*}
&|x+h-y|\leq |x-y|+|h|\leq Cr+r\leq (C+1)(r+\sigma),
\\ &|x-h-y|\leq |x-y|+|h|\leq Cr+r\leq 2(C+1)(r-\sigma).
\end{align*}
We also know that $(\eta,v,l)|_{[0,r]}$ is a minimizer of $S_r(x,y)$. Thus, by Lemmas \ref{SleqA} and \ref{SgeqA},
\begin{align*}
&S_{r+\sigma}(x+h,y)\leq A_{r+\sigma}(x+h,y)+C(r+\sigma)^3,
\\ &S_{r-\sigma}(x-h,y)\leq A_{r-\sigma}(x-h,y)+C(r-\sigma)^3,
\\ &S_r(x,y)\geq A_r(x,y)-Cr^3.
\end{align*}
Here, since we take $r>\sigma$,
\[(r+\sigma)^3+(r-\sigma)^3+2r^3=6r\sigma^2+4r^3\leq 10 r^3.\]
By the time-reversal transformation and applying \cite[Proposition B.3]{CWC2}, we have
\[A_{r+\sigma}(x+h,y)+A_{r-\sigma}(x-h,y)-2A_r(x,y)\leq C\frac{(|h|+\sigma)^2}{r}.\]
Optimizing
\[r^3+\frac{(|h|+\sigma)^2}{r},\]
we get
\[r=(|h|+\sigma)^{\frac{1}{2}}.\]
Here, if we take $|h|+\sigma\leq \delta$ with $\delta>0$ small enough, we have
\[|h|\leq r,\quad \sigma\leq \frac{r}{2},\quad r\leq\frac{\ep}{2}.\]
Following \cite[Proposition 3.1]{MN1}, one can prove the Lipschitz continuity of $u(x,t)$ on $\ol\Omega\times[0,\infty)$. For $\sigma>0$ such that $t-\sigma\geq 0$, there exists $C>0$ independent of $t$ such that
\[|u(x\pm h,t\pm \sigma)-u(x,t)|\leq C(|h|+\sigma).\]
If $|h|+\sigma\geq \delta$, we have
\[2C(|h|+\sigma)\leq \frac{2C}{\sqrt{\delta}}(|h|+\sigma)^{\frac{3}{2}},\]
where $\delta$ depends on $\ep$. We conclude
\[u(x+h,t+\sigma)+u(x-h,t-\sigma)-2u(x,t)\leq C(|h|+\sigma)^{\frac{3}{2}},\]
for some constant $C>0$ depending on $\varepsilon$, $\|Du_0\|_{L^\infty(\R^n)}$, $\|g\|_{\Lip(\R^n)}$, $H$, and $\Omega$.
\qed

\medskip

\begin{remark}\label{rmk1}
For the state-constraint problem, we consider
\[S^0_\tau(x,y)=\inf \Big\{\int_0^\tau L(\eta(s),-\dot\eta(s))\, ds\mid \eta\in\AC([0,\tau],\ol\Omega),\ \eta(0)=x,\ \eta(\tau)=y\Big\}.\]
Then the inequality
\[S^0_\tau(x,y)\geq A_\tau(x,y)\]
is immediate from the inclusion of admissible classes. Since we take $l(s)\equiv 0$ in the proof of Lemma \ref{SleqA}, we still have
\[S^0_\tau(x,y)\leq A_\tau(x,y)+C\tau^3.\]
Here we note that Lemma  \ref{SleqA} does not require the structural assumption {\rm (A3)}. This suggests an alternative approach to semiconcavity estimates for state-constraint problems based on the comparison between constrained and unconstrained actions, rather than on higher-order regularity properties of minimizing trajectories.
\end{remark}

\section{Examples of non-$C^1$ minimizing trajectories}\label{sec:half}

In this section, we present examples showing that minimizing trajectories of the Skorokhod problem may fail to be $C^1$. This contrasts sharply with the state-constraint case, where the semiconcavity theory relies on the $C^{1,1}$-regularity of minimizers. These examples motivate the different approach adopted in the proof of Theorem \ref{thm:main}.

\subsection{Case 1: $g<0$ and (A3) holds}\label{sec:4.1}

Consider
\[\Omega:=\{z=(z',z_n)\in\R^{n-1}\times\R:\ z_n>0\},\]
and
\[
H(p):=\frac{1}{2}|p|^2,\quad g(x):=-a,\quad u_0:=0
\]
for some constant $a>0$. 
In this setting, the solution is given by 
\begin{equation}\label{sol:ex1}
u(x,t)=\inf\left\{
J[\eta,v,l](x,t) \mid (\eta,v,l)\in\SP(x)\right\}, 
\end{equation}
for all $(x,t)\in\overline{\Omega}\times[0,\infty)$,  
where 
$J[\eta,v,l](\cdot,t):=\int_0^t \frac{1}{2}\left(|v(s)|^2-al(s)\right)\,ds$.   
Let $e_n$ denote the $n$-th standard basis vector of $\R^n$. Note that $\nu=-e_n$ on $\partial\Omega$. We take
\[x=r e_n\quad \text{for some} \ r>0.\]
%There are two types of curves which are candidates for minimizers. One is the constant curve $\eta(s)=x$ and $l=0$ since $x\in\Omega$. The action is
%\[\int_0^t \Big(\frac{1}{2}|\dot\eta(s)|^2-al(s)\Big)\, ds+u_0(\eta(t))=0.\]
%The other one is the curve which starts at $x$ and moves directly in the $-e_n$-direction and hits the boundary at the finite time $\tau$. For $\eta(s)\in\partial \Omega$,
By Jensen's inequality, minimizing trajectories of $\int_0^t\frac{1}{2}|\dot\eta(s)|^2\, ds$ are straight lines with constant
speed. Here, we consider a curve which starts at $x$ and moves straightly to the $-e_n$-direction with a constant speed, and hits the boundary at a finite time $\tau>0$, which will be determined later. 
If $(\eta,v,l)\in\SP(x)$ is a minimizer, by Lemma \ref{lem:dot}, we have that for $\eta(s)\in\partial \Omega$,
\[J[\eta,v,l](\cdot,t)=\int_{I_0^c}\left(\frac{1}{2}|v(s)|^2+g(\eta(s))l(s)\right)\, ds=\int_{I_0^c}\left(\frac{1}{2}|\dot\eta(s)|^2+\frac{1}{2}l(s)^2-al(s)\right)\, ds,\]
therefore a minimizer should satisfy
\[
\dot\eta(s)=0,\quad l(s)=a \qquad\text{for} \ s\in[\tau,\infty).
\]
Note that before $\eta$ hits the boundary, since the speed is constant, we have $|\dot{\eta}(s)|=\frac{r}{\tau}$ for $s\in[0,\tau]$. After hitting the boundary, $\eta$ must stay on the boundary, otherwise, the action increases. Therefore, 
the action is
\[
\int_0^\tau\frac{r^2}{2\tau^2}\, ds+\int_\tau^{t}\Big(-\frac{a^2}{2}\Big)\,ds=\frac{r^2}{2\tau}-\frac{a^2}{2}(t-\tau),
\]
and minimizing the above on $\tau$, we obtain 
\[\tau=\frac{r}{a}.\]
Therefore, 
\begin{equation}\label{ex1:sol}
J[\eta,v,l](x,t)=ar-\frac{a^2}{2}t.
\end{equation}
We can easily check that for all $x=(x',x_n)\in R^{n-1}\times\R$ and $t>0$, 
\[
u(x,t)
=\min\left\{0,\,
    ax_n-\frac{a^2}{2}t\right\}
\]
gives the solution to \eqref{eq:N} in this setting. 
Therefore, we see that $(\eta,v,l)$ given by 
\begin{equation*}
\eta(s):=
\begin{cases}
(r-as)e_n\quad & \text{for} \ s\in [0,\frac{r}{a}],\\
0\quad & \text{for} \ s\in [\frac{r}{a},\infty), 
\end{cases}
\quad 
l(s):=
\begin{cases}
0\quad & \text{for} \ s\in [0,\frac{r}{a}],\\
a\quad & \text{for} \ s\in [\frac{r}{a},\infty), 
\end{cases}
\end{equation*}
and $v=\dot{\eta}-le_n=-ae_n$ for all $s\in[0,\infty)$ is a minimizer of \eqref{sol:ex1} for $x=re_n$ and $t>\frac{2r}{a}$. 
Clearly, the curve $\eta(s)$ is not $C^1$ at $s=\frac{r}{a}$, since $\dot\eta(s)$ jumps from $-ae_n$ to zero.

\subsection{Case 2: (A3) fails and $g>0$}

Consider
\[\Omega:=\{z=(z',z_n)\in\R^{n-1}\times\R:\ z_n>0\},\]
and
\[
H(p):=\frac{1}{2}|p+\beta e_n|^2,\quad g:=a,\quad u_0:=0
\]
for two constants $\beta>a>0$. Then, 
$L(v)=\frac{1}{2}|v|^2-\beta v_n$. 
Noting that 
\[
H(p)=\frac12 |(p_\tau,p_\nu)+\beta e_n|^2=\frac12|p_\tau|^2+\frac12|p_\nu-\beta\nu|^2, 
\]
we can check that {\rm (A3)} fails. We take
\[x=r e_n,\quad r>0.\]
When $\eta(s)\notin \partial\Omega$, we have $l(s)=0$ and
\[L(-\dot\eta(s))=\frac{1}{2}|\dot\eta'(s)|^2+\frac{1}{2}|\dot\eta_n(s)|^2+\beta \dot\eta_n(s),\]
where $\eta(s)=(\eta'(s),\eta_n(s))$. So, a minimizer should satisfy 
\[\dot\eta'(s)=0,\quad \dot\eta_n(s)<0.\]
Assume that $\eta$ hits the boundary at time $\tau$. Then, 
\begin{align*}
\int_0^\tau\Big(\frac{1}{2}|\dot\eta_n(s)|^2+\beta \dot\eta_n(s)\Big)\, ds&=\int_0^\tau \frac{1}{2}|\dot\eta_n(s)|^2\, ds+\beta (\eta_n(\tau)-\eta_n(0))
\\ &=\int_0^\tau \frac{1}{2}|\dot\eta_n(s)|^2\, ds-\beta r.
\end{align*}
So, before $\eta$ hits the boundary, the action is minimized by the straight line with constant speed $\frac{r}{\tau}$. Since $\dot\eta_n(s)<0$ if $\eta(s)\notin\partial\Omega$, once $\eta(s)$ is on $\partial\Omega$, any outward normal velocity is forbidden and any tangential velocity increases the quadratic part of the action. Hence, the minimizing boundary motion has $\dot\eta(s)=0$. We recall that on the boundary, we have $\nu=-e_n$. Therefore, for $\eta(s)\in\partial\Omega$, we have
\[L(-\dot\eta(s)-\nu(\eta(s)) l(s))+g(\eta(s))l(s)=L(l(s)e_n)+al(s)=\frac{1}{2}l(s)^2+(a-\beta) l(s).\]
Optimizing the above on $l$, we obtain $l(s)=\beta-a$. 
The action is then
\[\int_0^\tau \frac{1}{2}\Big(\frac{r}{\tau}\Big)^2\, ds-\beta r-\int_\tau^t\frac{1}{2}(\beta-a)^2\, ds=\frac{r^2}{2\tau}-\beta r-\frac{(\beta-a)^2}{2}(t-\tau).\]
Minimizing the above on $\tau$, we obtain 
$\tau=\frac{r}{\beta-a}$, 
and therefore the action is 
\[-ar-\frac{(\beta-a)^2}{2}t.\]
Clearly, this gives a solution to \eqref{eq:N} in this setting for $x=re_n$ and $t\ge\frac{r}{\beta-a}$, 
and therefore we see that $(\eta,v,l)$ given by 
\begin{equation*}
\eta(s):=
\begin{cases}
(r-(\beta-a)s)e_n\quad & \text{for} \ s\in [0,\frac{r}{\beta-a}],\\
0\quad & \text{for} \ s\in [\frac{r}{\beta-a},\infty), 
\end{cases}
\end{equation*}
and
\begin{equation*}
l(s):=
\begin{cases}
0\quad & \text{for} \ s\in [0,\frac{r}{\beta-a}],\\
\beta-a\quad & \text{for} \ s\in [\frac{r}{\beta-a},\infty), 
\end{cases}
\quad
v(s):=-(\beta-a)e_n \quad\text{for} \ s\in[0,\infty)
\end{equation*}
is a minimizer of the value function associated to this problem. 
The curve $\eta(s)$ is clearly not $C^1$ at $s=\frac{r}{\beta-a}$, since $\dot\eta(s)$ jumps from $-(\beta-a)e_n$ to zero.

\begin{remark}\label{rmk2}
The examples above show that minimizing trajectories associated with the
Skorokhod problem may fail to be $C^1$ in general. On the other hand, if $g(x)\geq 0$ for all $x\in\partial\Omega$ and {\rm (A3)} holds, then any minimizer of \eqref{vf} satisfies $l(s)\equiv 0$. Indeed,
\[L(\eta(s),-\dot\eta(s)-l(s)\nu(\eta(s)))+g(\eta(s))l(s) \geq L(\eta(s),-\dot\eta(s))\quad\text{for a.e. }s\in I^c_0,\]
and therefore the reflection term does not make the action smaller. Then the problem reduces to the state constraint problem studied in \cite{CCC2}. The examples presented above do not apply to the situation considered in Theorem \ref{thm:main}.
When $x\in \Omega_+$, $\partial\Omega_-\neq \emptyset$ and {\rm (A3)} holds, it remains open whether minimizing trajectories of $u(x,t)$ are $C^1$ or $C^{1,1}$. One difficulty is that the contact set $\partial I_0$ between the trajectory and the boundary may exhibit a complicated structure, making it unclear whether variational arguments can still be applied.
\end{remark}

\section{Optimality of the fractional exponent}\label{sec:ex1}

The main goal of this section is to show that the fractional exponent $\frac{3}{2}$ in Theorem \ref{thm:main} is optimal. The same example also shows that the exponent appearing in the state-constraint framework \cite{CCC1,CCMW} cannot be improved, since in this example the reflection term $l(s)$ vanishes and the dynamics reduce to the state-constraint case.

We consider the following Neumann boundary problem
\begin{equation}\label{lse}
\begin{cases}
u_t+|Du|^2=0\quad \text{in}\quad \Omega\times(0,\infty),
\\ Du\cdot \nu=0 \quad \text{on}\quad \partial\Omega\times(0,\infty),
\\ u(x,0)=u_0(x) \quad\text{on} \ \overline{\Omega}, 
\end{cases}
\end{equation}
where $\Omega=\mathbb R^2\setminus \overline{B(0,1)}$ and for $x=(x_1,x_2)\in\mathbb R^2$,
\begin{equation*}
u_0(x_1,x_2)=
\begin{cases}
x_1+2\quad &\text{for }x_1<0,
\\ 2\quad &\text{for }x_1\geq 0.
\end{cases}
\end{equation*}
It is obvious that \eqref{lse} satisfies all our assumptions. A related example is examined in \cite[Section 2.1]{MN1} in order to illustrate the detailed behavior of solutions to the Neumann boundary value problem in the sense of viscosity solutions. Related front propagation problems with Neumann boundary conditions were also studied in \cite{CM}. Since $g\equiv 0$, the representation formula reads
\begin{align*}
u(x,t)
&=
\inf_{(\eta,v,l)\in\SP(x)}
\left\{
u_0(\eta(t))
+\frac14\int_0^t|v(s)|^2\,ds
\right\}
\\ &=\inf_{\eta\in\AC([0,t],\ol\Omega)}
\left\{
u_0(\eta(t))
+\frac14\int_0^t|\dot\eta(s)|^2\,ds
\right\}.
\end{align*}
By Jensen's inequality, every minimizing trajectory is a shortest path
traversed with constant speed. Set
\begin{align*}
&\Omega_1:=\bar\Omega\cap \Big\{(x_1,x_2)\in\R^2\mid x_1>0,\ 0\leq x_2\leq 1\Big\},
\\ &\Omega_2:=\bar\Omega\cap \Big\{(x_1,x_2)\in\R^2\mid x_1>0,\ -1\leq x_2\leq 0\Big\}.
\end{align*}
Let
\[
\Sigma
=
\overline\Omega
\cap
\{x_1=0\},
\]
and denote by
$d(x)$ and $d(x,y)$
the intrinsic distance from
$x\in\ol\Omega$
to
$\Sigma$ and to $y\in\ol\Omega$
 inside
$\overline\Omega$, respectively. Then
\begin{equation}\label{u=d}
u(x,t)=\inf_{y\in\ol\Omega}\left\{u_0(y)+\frac{d(x,y)^2}{4t}\right\}.
\end{equation}
Let us consider the polar coordinate $(r,\phi)$, that is, 
\[r=\sqrt{x_1^2+x_2^2},\quad \phi\in(-\pi,\pi] \text{ is chosen so that }x_1=r\cos\phi\text{ and } x_2=r\sin\phi.\]
In light of \cite[Section 2.1]{MN1}, for
$x\in \Omega_1$,
\[
d(x)
=
\frac\pi2
-\arccos\frac1r
-\phi
+\sqrt{r^2-1},
\]
and for $x\in\Omega_2$,
\[
d(x)
=
\frac\pi2
-\arccos\frac1r
+\phi
+\sqrt{r^2-1}.
\]
Let $y=(y_1,y_2)\in\ol\Omega$ be a minimal point of \eqref{u=d}. If $y_1\geq 0$, then $u_0(y)=2$ and
\[u_0(y)+\frac{d(x,y)^2}{4t}\geq 2\]
with equality if and only if $y=x$.
If $y_1<0$, since $u_0(y)=y_1+2$ is independent of $y_2$, for each $y_1$, it is enough to minimize $d(x,y)$ in \eqref{u=d} over all $y_2$. By the geometric characterization, see \cite[Section 2.1]{MN1}, for $x\in\Omega_1$ and each $y_1<0$, $y_2\mapsto d(x,(y_1,y_2))$ is minimized at $y_2=1$. Therefore, for $x\in \Omega_1$, we have $y=(-s,1)$ for some $s>0$. The shortest path consists of a shortest path from $x$ to $(0,1)$, followed by the horizontal segment joining $(0,1)$ and $(-s,1)$. Hence, $d(x,y)=d(x)+s$. Then
\begin{equation}\label{u0+d}
u_0(y)+\frac{d(x,y)^2}{4t}=2-s+\frac{(d(x)+s)^2}{4t}.
\end{equation}
There exists $t_0>0$ such that $t_0>d(x)$ for all $x\in\Omega_1$ with $x_1\leq 2$. Since we are only interested in the behavior of the solution near the boundary, we take $t\geq t_0$. The minimal point of \eqref{u0+d} is $s_*=2t-d(x)>0$ and
\[2-s_*+\frac{(d(x)+s_*)^2}{4t}=2-t+d(x)<2.\]
Then the solution to \eqref{lse} can be expressed explicitly as 
\[u((r,\phi),t)=
2-t+\frac{\pi}{2}-\arccos\frac{1}{r}-\phi+\sqrt{r^2-1}\]
in this region. Calculating directly, if $u$ is twice differentiable at $(r,\phi)$, we get
\[
u_r((r,\phi),t)=
\frac{\sqrt{r^2-1}}{r}\]
and
\[
u_{rr}((r,\phi),t)=
\frac{1}{r^2\sqrt{r^2-1}}.\]
This means that the semiconcavity of $u$ with a linear modulus cannot hold when $r=1$, i.e., on the boundary $\partial\Omega$. Moreover, for $(x_1,x_2)\in \inter\Omega_1$, letting $(1+h,\phi)$ be the polar coordinate of $(x_1,x_2)$, where $h>0$ is small enough, we consider 
\begin{align*}
&u((1+2h,\phi),t)+u((1,\phi),t)-2u((1+h,\phi),t)
\\ &=-\arccos\frac{1}{1+2h}+2\arccos\frac{1}{1+h}+\sqrt{(1+2h)^2-1}-2\sqrt{(1+h)^2-1}.
\end{align*}
A simple computation yields 
\[
\arccos\frac{1}{1+\alpha}=\sqrt{2}\alpha^{\frac{1}{2}}-\frac{5\sqrt{2}}{12}\alpha^{\frac{3}{2}}+o(\alpha^{\frac{3}{2}})
\quad\text{as} \alpha\to0.
\]
We obtain
\[-\arccos\frac{1}{1+2h}=-2h^{\frac{1}{2}}+\frac{5}{3}h^{\frac{3}{2}}+o(h^{\frac{3}{2}}),\]
\[2\arccos\frac{1}{1+h}=2\sqrt{2}h^{\frac{1}{2}}-\frac{5\sqrt{2}}{6}h^{\frac{3}{2}}+o(h^{\frac{3}{2}}),\]
as $h\to0$, and
\begin{align*}
&\sqrt{(1+2h)^2-1}-2\sqrt{(1+h)^2-1}
=2\sqrt{h}\Big(\sqrt{h+1}-\sqrt{h+2}\Big)
\\ 
=&\, 2\sqrt{h}\Big(1+\frac{h}{2}-\sqrt{2}\Big(1+\frac{h}{4}\Big)\Big)+o(h^{\frac{3}{2}})
=2(1-\sqrt{2})h^{\frac{1}{2}}+\Big(1-\frac{\sqrt{2}}{2}\Big)h^{\frac{3}{2}}+o(h^{\frac{3}{2}}).
\end{align*}
Summing up, we obtain 
\[
u((1+2h,\phi),t)+u((1,\phi),t)-2u((1+h,\phi),t)=\frac{8-4\sqrt{2}}{3}h^{\frac{3}{2}}+o(h^{\frac{3}{2}})
\]
as $h\to0$, 
which shows the optimality of the power $\frac{3}{2}$ obtained in Theorem \ref{thm:main}.

\section*{Acknowledgements}

The authors would like to thank Professor Diogo Gomes and Professor Hung V. Tran for several helpful comments and suggestions. 
The work of HM was partially supported by the JSPS grants: 
KAKENHI \#26K06861, \#25K07072, \#24K00531, \#26H02001. 
The work of PN was supported by the JSPS grant: KAKENHI \#26KF0103.

\section*{Declarations}

\noindent {\bf Conflict of interest statement:} The authors state that there is no conflict of interest.

\medskip

\noindent {\bf Data availability statement:} Data sharing not applicable to this article as no datasets were generated or analysed during the current study.


\begin{thebibliography}{}\label{sec:TeXbooks}

\bibitem{BL} 
G. Barles, P.-L. Lions, 
\emph{Fully nonlinear Neumann type boundary conditions for first-order Hamilton-Jacobi equations}, 
Nonlinear Analysis, Theory, Methods \& Applications, 16 (1991), 143 -- 153.

\bibitem{BDLS} G. Barles, F. Da Lio, P.-L. Lions, P. E. Souganidis, \emph{Ergodic problems and periodic homogenization for fully nonlinear equations in half-space type domains with Neumann boundary conditions}, Indiana Univ. Math. J. 57 No. 5 (2008), 2355--2376.

\bibitem{CCC2} P. Cannarsa, R. Capuani, P. Cardaliaguet, \emph{$C^{1,1}$-smoothness of constrained solutions in the calculus of variations with application to mean field games}. Math. Eng. 1(1), 174--203 (2018).

\bibitem{CCC1} P. Cannarsa, R. Capuani, P. Cardaliaguet, \emph{Mean field games with state constraints: from mild to pointwise solutions of the PDE system}. Calc. Var. Part. Differ. Equ. 60(3), 33 (2021).

\bibitem{CCMW} P. Cannarsa, W. Cheng, C. Mendico, and K. Wang, \emph{Weak KAM Approach to First-Order Mean Field Games with State Constraints},  J. Dyn. Diff. Equat., 35, 1885--1916 (2023).

\bibitem{CWC} P. Cannarsa, W. Cheng, \emph{Singularities of Solutions of Hamilton--Jacobi Equations}, Milan J. Math., 89, 187--215 (2021).

\bibitem{CWC2}
P. Cannarsa, W. Cheng, \emph{Generalized characteristics and Lax-Oleinik operators: global theory}, Calc. Var. Part. Differ. Equ., 56, 125 (2017).

\bibitem{CM} P. Cardaliaguet, C. Marchi, \emph{Regularity of the eikonal equation with Neumann boundary conditions in the plane: application to fronts with nonlocal terms}, SIAM J. Control Optim. 45(3), 1017--1038 (2006).

\bibitem{CS} P. Cannarsa, C. Sinestrari, 
\emph{Semiconcave Functions, Hamilton--Jacobi Equations and Optimal Control},
Progress in Nonlinear Differential Equations and their Applications, vol. 58. Birkh\"auser Boston Inc., Boston (2004).

\bibitem{CY} P. Cannarsa, Y. Yu, \emph{Singular dynamics for semiconcave functions}, J. Eur. Math. Soc., 11(5), 999--1024 (2009).

\bibitem{guide}
M. Crandall, H. Ishii, P.-L. Lions,
\emph{User's guide to viscosity solutions of second order partial differential equations},
Bull. Am. Math. Soc. (N.S.) 27 (1992) 1 -- 67.

\bibitem{Day}
M. Day, \emph{Neumann-type boundary conditions for Hamilton-Jacobi equations in smooth domains}, Appl. Math. Optim. 53 (2006), no. 3, 359--381. 

\bibitem{D}
A. Douglis, 
\emph{The continuous dependence of generalized solutions of non-linear partial differential equations upon initial data}, 
Comm. Pure Appl. Math. 14 (1961), 267--284.

\bibitem{F} A. Fathi, \emph{Weak KAM Theorems in Lagrangian Dynamics}, Cambridge: Cambridge University Press, 10th preliminary version, 2008.

\bibitem{Han}
Y. Han,
\emph{Global semiconcavity of solutions to first-order Hamilton-Jacobi equations with state constraints},
Adv. Cont. Discr. Mod., 2025, 106 (2025).

\bibitem{I11} H. Ishii, 
\emph{Weak KAM aspects of convex Hamilton--Jacobi equations with Neumann type boundary conditions}, 
J. Math. Pures Appl., 95 (2011), 99 -- 135.

\bibitem{I13} H. Ishii,
\emph{A short introduction to viscosity solutions and the large time behavior of solutions of Hamilton--Jacobi equations}. In: Hamilton--Jacobi Equations: Approximations, Numerical Analysis and Applications. Lecture Notes in Mathematics, Springer, Berlin, Heidelberg, (2013).

\bibitem{K}
S. N. Kru\v{z}kov, \emph{Generalized solutions of Hamilton-Jacobi equations of eikonal type. I. Statement of the problems; existence, uniqueness and stability theorems; certain properties of the solutions} (Russian), Mat. Sb. (N.S.) 98(140) (1975), no. 3(11), 450--493, 496. 

\bibitem{LS}
P.-L. Lions, A.-S. Sznitman, 
\emph{Stochastic differential equations with reflecting boundary conditions}, 
Comm. Pure Appl. Math. 37 (1984), no. 4, 511--537. 

\bibitem{L}
P.-L. Lions. \emph{Neumann type boundary conditions for Hamilton-Jacobi equations}, 
Duke Math. J., 52(4):793--820, 1985.

\bibitem{MN1} H. Mitake, P. Ni, \emph{Quantitative homogenization of convex Hamilton--Jacobi equations with Neumann
type boundary conditions}, Calc. Var. Part. Differ. Equ., (2026) 65:154.

\bibitem{MNT} H. Mitake, P. Ni, H. V. Tran, \emph{Quantitative homogenization of convex Hamilton--Jacobi equations with $u/\ep$-periodic Hamiltonians}, Preprint is available from arXiv:2507.00663.

\bibitem{T} H.V. Tran, \emph{Hamilton--Jacobi Equations—Theory and Applications}, Graduate Studies in Mathematics, vol. 213, American Mathematical Society, 2021.

\bibitem{Y}Y. Yu, \emph{A simple proof of the propagation of singularities for solutions of Hamilton--Jacobi equations}, Ann. Sc. Norm. Super. Pisa Cl. Sci. (5) 5(4), 439--444 (2006).

\end{thebibliography}
\end{document}